\documentclass[amsfonts,twoside]{article}
\usepackage{amsmath,amsfonts,amsthm,graphicx,amscd,amssymb,latexsym,graphics,color,longtable,multicol,verbatim,graphics,graphicx,lscape}

\usepackage[linesnumbered,ruled,vlined]{algorithm2e}%Produce algorithm enviroment

 \font\chuto=cmbx10 at 14pt
\font\chudaude=cmcsc10

\font\ita=cmti9

\title{\vspace{-2cm}
\vspace{2cm} {\chuto An algorithm for determining Lie algebra types
}\footnotetext{\\
{$^*$ Corresponding author
\\
\bf Key
words:} Lie algebra, Jordan–Kronecker invariant.   
\\
2000 AMS Mathematics Subject Classification: Primary 17B05, Secondary 15A21; 68W30.
   } }
\author{{\bf   Tu T. C. Nguyen$^1$, Tuan A. Nguyen$^2$, Vu A. Le$^{3,*}$}\\
\vspace{-0.2cm}{\ita    
} \\
\vspace{-0.2cm}{\ita  $^1$College of Natural Sciences, Can Tho University}\\
\vspace{-0.2cm}{\ita    E-mail: camtu@ctu.edu.vn
}\\
\vspace{-0.2cm}{\ita   $^2$Department of Primary Education, Ho Chi Minh City University of Education}\\
\vspace{-0.2cm}{\ita    E-mail: tuannguyenanh@hcmue.edu.vn
}\\
\vspace{-0.2cm}{\ita $^3$Department of Economic Mathematics, University of Economics and Law}  
\\
\vspace{-0.2cm}{\ita Vietnam National University, Ho Chi Minh City}  
\\
\vspace{-0.2cm}{\ita    E-mail: vula@uel.edu.vn
}\\
 }
 
\newcommand{\A}{\mathcal{A}}%
\newcommand{\B}{\mathcal{B}}%
\newcommand{\Pp}{\mathcal{P}}%
\newcommand{\rank}{\mathrm{rank\,}}%
\newcommand{\ind}{\mathrm{ind\,}}%

\newcommand{\G}{\mathcal{G}}%

\begin{document}
\date{ }
\maketitle
 \label{firstpage}

\newtheorem{theorem}{Theorem}[section]
\newtheorem{prove}{Proof of Theorem}[section]
\newtheorem{lemma}[theorem]{Lemma}
\newtheorem{proposition}[theorem]{Proposition}
\theoremstyle{definition}
\newtheorem{definition}{Definition}[section]
\newtheorem{remark}[definition]{Remark}
\newtheorem{example}[definition]{Example}
\newtheorem*{notation}{Notation}

\newcommand{\map}[3]{{#1} :{#2} \longrightarrow {#3}}
\newcommand{\mapp}[5]{\begin{array}{rl}
{#1} :{#2} & \longrightarrow {#3} \\
{#4} & \longmapsto {#5}
\end{array}}
 %%%%%%%%%%%%%%%%%%%%%%%
  
\begin{abstract}
This paper investigates the Jordan--Kronecker invariant of finite dimensional complex Lie algebras. We present an explicit algorithm for determining the type of a given Lie algebra from its Jordan--Kronecker invariant. The algorithm is implemented in a specific Matlab program.
\end{abstract}
 
\section*{\bf Introduction}

The study of invariants of Lie algebras plays a fundamental role in understanding their structure and classification. Among these invariants, the Jordan–Kronecker invariant, which was introduced by A. V. Bolsinov and P. Zhang \cite{Bol}, is an effective tool for describing the canonical form of the generic pencil associated with a Lie algebra. Given a finite-dimensional complex Lie algebra $\G$, its Jordan–Kronecker invariant captures the number and sizes of the Jordan and Kronecker blocks in the Jordan–-Kronecker decomposition of the generic pencil. Moreover, this invariant provides an effective mechanism for classifying Lie algebras according to their Jordan and Kronecker types.

In this paper, we address the problem of determining the type of a finite-dimensional complex Lie algebra from its Jordan–Kronecker invariant. Building on structural properties established in \cite{Bol}, we present an explicit algorithm that identifies whether the generic pencil contains Jordan blocks, Kronecker blocks, or both. The algorithm is fully implemented in a specific Matlab program.

The structure of the paper is as follows. In Section \ref{sec2}, we recall the necessary background on pencils of skew-symmetric matrices and the Jordan--Kronecker invariants. Section \ref{sec3} presents the main algorithm together with its Matlab implementation, as well as illustrative examples. In the final section, we present a table summarizing the types of 7-dimensional complex indecomposable solvable Lie algebras having 5-dimensional nilradicals. The results are obtained by applying the algorithm to each algebra in this class.

%%%%%%%%%%%%%%%%%%%%%%%%%%%%%%%%%%%%%%%%%%%%%%%%%%%%%%%%%%%%%%%%%%%%%%%%%%%%%%%
\section{Basic definitions and theorems}\label{sec2}
This section briefly reviews some notions and results related to Jordan--Kronecker invariants of Lie algebras. For a more detailed exposition, see \cite{Bol}.

\begin{theorem}[{\cite[Theorem 2]{Bol}}]\label{JKTheorem}
    Let $\A$ and $\B$ be two skew-symmetric bilinear forms on a complex vector space $V$. Then by an appropriate choice of a basis, their matrices can be simultaneously reduced to the following canonical block-diagonal form:
    \[
    \A \mapsto \begin{pmatrix}
    \A_1 & & \\
    & \ddots & \\
    & & \A_k
    \end{pmatrix}, \quad
    \B \mapsto \begin{pmatrix}
    \B_1 & & \\
    & \ddots & \\
    & & \B_k
    \end{pmatrix}
    \]
    where the pairs of the corresponding blocks $\A_i$ and $\B_i$ can be of the following three types:
    
    {\center \includegraphics[scale=0.85]{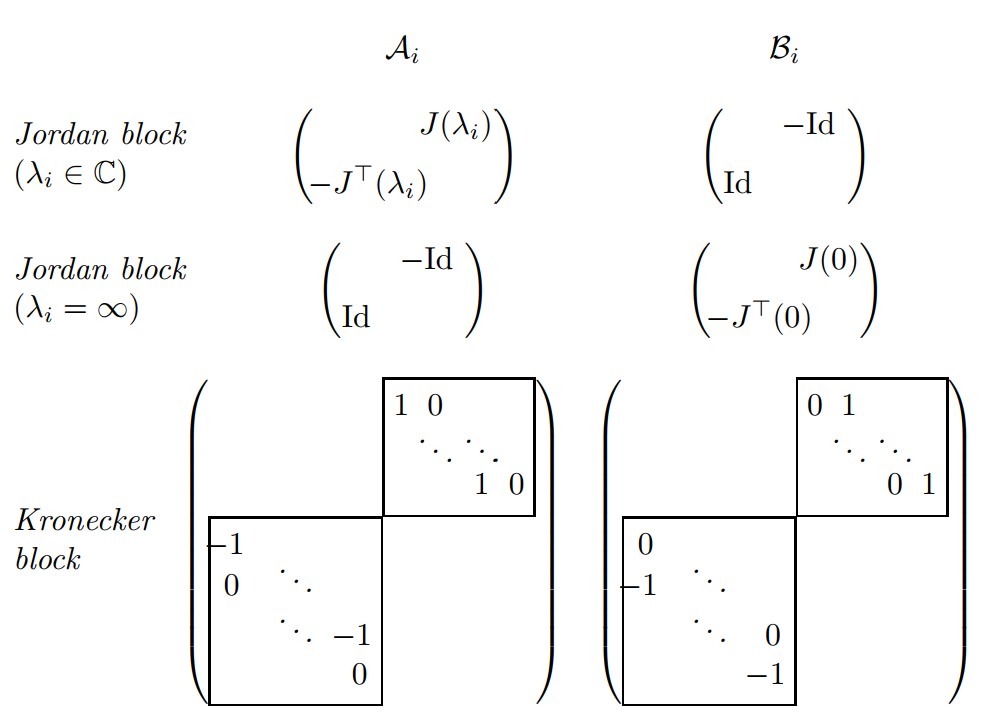}\\}

    \noindent where $J(\lambda_i)$ denotes the standard Jordan block
    \[
    J(\lambda_i) = 
    \begin{pmatrix}
    \lambda_i & 1 & & \\
    & \lambda_i & \ddots & \\
    & & \ddots & 1 \\
    & & & \lambda_i
    \end{pmatrix}.
    \]
\end{theorem}

\begin{remark}
The blocks $\A_i$ and $\B_i$ of the canonical form given in Theorem \ref{JKTheorem} are defined uniquely up to permutation.
\end{remark}

\begin{remark}
    Consider the pencil \( \Pp=\{\A  + \lambda \B\} \) generated by skew-symmetric bilinear forms $\A$, $\B$. Then
    \begin{itemize}
        \item The {\em rank} of \( \Pp \) is defined as 
    \[
    \operatorname{rank} \Pp = \max_\lambda \operatorname{rank} (\A  + \lambda \B).
    \]
    \item The numbers \( \lambda_i \) that appear in the Jordan blocks \( \A_i \) of the canonical form given in Theorem \ref{JKTheorem} are called {\em characteristic numbers} of \( \Pp \). They are those numbers for which the rank of \( (\A  + \lambda \B) \) with 
    \( \lambda = \lambda_i \) is not maximal, i.e.
    \[
    \operatorname{rank} (\A  + \lambda_i \B) < \operatorname{rank} \Pp.
    \]

    \item By Theorem \ref{JKTheorem}, the pair of skew-symmetric bilinear forms $\A$, $\B$ can be simultaneously reduced to the canonical block-diagonal form. This result is also known as a {\em Jordan--Kronecker decomposition} which remains invariant under the replacement of \( \A, \B \) by any other pair \( (\A  + \lambda \B) \), \( (\A  + \mu \B) \) except for a change in the characteristic numbers by the transformation 
    \[ \lambda_i \mapsto \frac{\lambda_i - \lambda}{\mu - \lambda_i}. \]
    The number and sizes of Jordan and Kronecker blocks in the Jordan--Kronecker decomposition is called {\em algebraic type} of $\Pp$. We say that two pencils have the same algebraic type if they have the same Kronecker blocks and there is a one-to-one correspondence between their spectra such that the sizes of Jordan blocks for any corresponding characteristic numbers are the same. 
    \end{itemize}
        
\end{remark}

\begin{definition}
An $n$-dimensional Lie algebra $\G$ is an $n$-dimensional vector space $\G$ together with a skew-symmetric bilinear map
\[
\G \times \G \to \G ,\;\; (x,y) \mapsto [x,y],
\]
satisfying the Jacobi identity.

If the underlying field of the vector space $\G$ is the complex field then $\G$ is a complex Lie algebra. If we fix a basis $\{e_1,\,e_2,\, \dots,\,e_n\}$ on an $n$-dimensional Lie algebra $\G$, then
\[
[e_i,e_j]=\sum_{k=1}^n c_{i j}^ke_k,
\]
where the numbers $c_{i j}^k$ are called the structure constants of $\G$.
\end{definition}
    
   Throughout the rest of paper, let $\G$ be an $n$-dimensional complex Lie algebra and $\G^*$ be its dual space unless otherwise stated. Denote by $\{e_1,\dots,e_n\}$ the basis of $\G$. Let $\{e^1,\dots,e^n\}$ be the dual basis of this basis. For a pair of points \( x, a \in \G^* \), consider the skew-symmetric bilinear forms $\A_x = \left( \sum c_{ij}^k x_k \right)$, $\A_a = \left( \sum c_{ij}^k a_k \right)$ and the pencil \( \{\A_x  + \lambda \A_a\} \) generated by them. Here, recall that \( c^k_{ij} \) are the structure constants of \(\G\) and $x_k$, $a_k$ are the coordinates of $x$, $a$ with respect to the dual basis $e^k$, respectively. Clearly, the algebraic type of \( \{\A_x + \lambda \A_a\} \) essentially depends on the choice of \(x\) and \(a\). However, it remains ``constant'' almost everywhere.

    \begin{proposition}[{\cite[Proposition 1]{Bol}}]\label{Zar}
    There exists a non-empty Zariski open subset \( U \subset \G^* \times \G^* \) such that the algebraic type of the pencil \( \{\A_x + \lambda \A_a\} \) is the same for all \( (x, a) \in U \).
    \end{proposition} 

    Let \( U \) be a non-empty Zariski open subset from Proposition \ref{Zar}. We say that \( (x, a) \in U \) is a \emph{generic pair}. The corresponding pencil \( \{ \A_x + \lambda \A_a \} \) is also called \emph{generic}.

\begin{definition}[{\cite[Definition 1]{Bol}}]
    The algebraic type of a generic pencil \( \{\A_x + \lambda \A_a \} \) is called the \emph{Jordan-- Kronecker invariant} of \( \G \).
    In particular, we say that the Lie algebra \( \G \) is of
    \begin{itemize}
        \item \emph{Kronecker type},
        \item \emph{Jordan type},
        \item \emph{mixed type},
    \end{itemize}
    if the Jordan--Kronecker decomposition for the generic pencil 
        \(\{ \A_x + \lambda \A_a \}\) consists of
    \begin{itemize}
        \item only Kronecker blocks,
        \item only Jordan blocks,
        \item both Jordan and Kronecker blocks,
    \end{itemize}
    respectively.
\end{definition}

\section{An algorithm determining Lie algebra types}\label{sec3}

To determine whether Jordan or Kronecker blocks appear in the Jordan--Kronecker decomposition for the generic pencil \(\{ \A_x + \lambda \A_a \}\), we use the result stated in the following proposition.

 \begin{proposition}\label{xd}
\begin{enumerate}
    \item The number of Kronecker blocks equals the index of $\G$, denoted by $\ind \G$, in which
    \begin{equation}\label{indG}
        \ind \G= \dim \G - \max_{x \in \G^*} \operatorname{rank} \A_x.
    \end{equation}
    \item There are no Jordan blocks if the characteristic polynomial \( p(\lambda) \) is trivial, i.e. \( p(\lambda)=1 \), where \( p(\lambda) \) is determined as follow:
     \begin{itemize}
                \item Consider all diagonal minors of the matrix \( \A_x\) of order \( \rank \, \G = \max_{x \in \G^*} \operatorname{rank} \A_x\);
                \item Take the Pfaffians, i.e., square roots, for each of these minors;
                \item Let $p_0(x)$ be the greatest common divisor of all these Pfaffians;
                \item Then \( p(\lambda) = p_0(x+\lambda a)\).
            \end{itemize}
\end{enumerate}
 \end{proposition}

\begin{proof}
    This proposition is, in fact, one of the key results of~\cite{Bol}. Item (1) is a restatement of~\cite[Proposition 2, Item (1)]{Bol}. The assertion that “There are no Jordan blocks if the characteristic polynomial \( p(\lambda) \) is trivial’’ follows from the equivalence between statements (1) and (3)  of~\cite[Corollary 5]{Bol}. Finally, the method used here to determine \( p(\lambda) \) can be found in the introduction of the characteristic polynomial of~\cite[p.59]{Bol}. 
\end{proof}
 
Based on Proposition~\ref{xd}, we propose an algorithm to determine the type of a given Lie algebra as follows.

\begin{algorithm}[h]
	%\DontPrintSemicolon %Hiện dấu ;
	%\SetAlgoLined % Sử dụng end
	\KwIn{Structure constants $c_{ij}^k$ of an $n$-dimensional Lie algebra $\G$}
	\KwOut{The type of $\G$}
	Determine the matrix $A_x$, and then compute $\ind \G$ by formula~\eqref{indG}\;
	\eIf{$\ind \G =0$}
		{return $\G$ is of Jordan type}
		{
        Determine $p(\lambda)$\;
        \eIf{$p(\lambda) = 1$}
        {return $\G$ is of Kronecker type}
        {return $\G$ is of mixed type}
        }
	\caption{Determining the type of a Lie algebra}\label{alg1}
\end{algorithm}

%\begin{center}
%    \begin{tabular}{p{11.6cm}} \hline
%         \bf Algorithm.\; Determining the type of a Lie algebra from its Jordan-\\ \bf \quad \quad \quad \quad \quad \, Kronecker invariants  \\ \hline
%          \textbullet \; \textbf{Input:} the dimension $n$ and the structure constants of the Lie algebra $\G$;\\
%          \textbullet \;  Determine the matrix $A_x$;\\
%          \textbullet \;  Compute $\ind \G$:\\
%           \quad $-$ If $\ind \G =0$, \textbf{output:} $\G$ is of Jordan type.\\
%           \quad $-$ If $\ind \G \ne 0$, determine $p(\lambda)$:\\
%          \quad \quad $\circ$ If $p(\lambda) = 1$, \textbf{output:} $\G$ is of Kronecker type.\\
%          \quad \quad$\circ$ If $p(\lambda) \ne 1$, \textbf{output:} $\G$ is of mixed type.\\ \hline
%    \end{tabular}
%\end{center}

\begin{example}\label{Ex1}
    Consider the 4-dimensional Lie algebra $\G$ with non-trivial Lie brackets
    \[
    [e_3,e_1]=e_1,\; \; [e_3,e_4]=e_2.
    \]
    The non-zero structure constants of $\G$ are
    \[
    c_{13}^1=-1,\;\;c_{31}^1=1,\;\; c_{34}^2=1,\;\;c_{43}^2=-1.
    \]
    So we have 
    \[
    A_x = \begin{pmatrix}
        0&0&-x_1&0\\
        0&0&0&0\\
        x_1&0&0&x_2\\
        0&0&-x_2&0
    \end{pmatrix}.
    \]
    We obtain $\ind \G=2.$ Next, we need to determine $p(\lambda)$. Because all diagonal minors of order 2 of \( \A_x\) belong to the set $\{0,\, x_1^2,\, x_2^2\}$, it is easy to see that $p(\lambda)=1$. Thus, $\G$ is of Kronecker type.
\end{example}

We use \textsc{Matlab} to implement the above algorithm. The procedure is as follows.
\begin{enumerate}
    \item Input the dimension $n$ of the Lie algebra $\G$.
    \item Input the non-trivial Lie brackets of $\G$ to determine the matrix $A_x$ by running the program named ``\texttt{Lie\_brackets.m}''. Code of the program ``\texttt{Lie\_brackets.m}'' is
\begin{verbatim}
A = sym(zeros(n));
disp('Enter the non-trivial Lie brackets
(A(i,j) = [e_i, e_j], i < j): ')
\end{verbatim}

    \item Perform the determinations of the matrix $A_x$, the index of $\G$ and the characteristic polynomial $p(\lambda)$ in order to conclude the type of $\G$. This is executed by running the program named ``\texttt{Type\_of\_G.m}''. Code of the program ``\texttt{Type\_of\_G.m}'' is
\begin{verbatim}
for i=1:n
    for j=i+1:n
        A(j,i)=-A(i,j);
    end
end

r=rank(A);
indG=n-r;

idx = nchoosek(1:n, r);

minors = sym(zeros(size(idx, 1), 1));
pf_vals = sym(zeros(size(idx, 1), 1));

for k = 1:size(idx, 1)
    rows = idx(k, :);
    cols = idx(k, :); 
    subM = A(rows, cols);
    minors(k) = det(subM); 
    pf_vals(k) = sqrt(minors(k));
end

g = pf_vals(1);
for k = 2:length(pf_vals)
    g = gcd(g, pf_vals(k));
end

if indG==0
       disp('G is of Jordan type.')
   elseif g==1
       disp('G is of Kronecker type.')
   else
         disp('G is of mixed type.')
   end
\end{verbatim}
\end{enumerate}

\begin{example}
    For the Lie algebra $\G$ presented in Example \ref{Ex1}, the corresponding \textsc{Matlab} code is as follows.
    \begin{verbatim}
    >> n=4;
    >> syms x1 x2 x3 x4;
    >> Lie_brackets
    Enter the non-trivial Lie brackets
    (A(i,j) = [e_i, e_j], i < j):
    >> A(1,3)=-x1; A(3,4)=x2;
    >> Type_of_G
    G is of Kronecker type.
\end{verbatim}
The last line returned by the algorithm shows that $\G$ is of Kronecker type.
\end{example}

\section{Application to 7-dimensional Lie algebras}
In this section, we consider several 7-dimensional solvable Lie algebras with 5-dimensional nilradicals, classified by Vu et al. \cite{Vu}, to further illustrate the application of the algorithm presented in this paper. According to classification of Vu et al., the 7-dimensional complex Lie algebras listed in Table A.2 have the following structure.

\begin{table}[h!]
\centering
\begin{tabular}{|p{1cm}|p{10.1cm}|}
\hline
\textbf{Alg.} & \textbf{Non-trivial Lie brackets} \\
\hline

$L_{1}$ &
$[x_{1},x_{2}]=x_{3},\;
 [x_{1},x_{3}]=x_{4},\;
 [x_{1},x_{7}]= -x_{1},\;
 [x_{2},x_{7}]=2x_{2},$\\
 &$
 [x_{3},x_{7}]=x_{3},\;
 [x_{5},x_{6}] = -x_{5},\;
 [x_{6},x_{7}]=x_{4}$ \\ \hline

$L_{2}$ &
$[x_{1},x_{2}]=x_{3},\;
 [x_{1},x_{3}]=x_{4},\;
 [x_{1},x_{7}]=-x_{1},\;
 [x_{2},x_{6}]=-x_{2},$ \\
 &$
 [x_{3},x_{6}]=-x_{3},\;
 [x_{3},x_{7}]=-x_{3},\;
 [x_{4},x_{6}]=-x_{4},\;
 [x_{4},x_{7}]=-2x_{4},$ \\
 &$
 [x_{6},x_{7}]=x_{5}$ \\ \hline
 
 $L_{3}^{a}$ &
 $[x_{1},x_{2}]=x_{3},\;
 [x_{1},x_{3}]=x_{4},\;
 [x_{1},x_{7}]=-ax_{1},\;
 [x_{2},x_{6}]=-x_{2},$\\
 &$
 [x_{3},x_{6}]=-x_{3},\;
 [x_{3},x_{7}]=-ax_{3},\;
 [x_{4},x_{6}]=-x_{4},\;
 [x_{4},x_{7}]=-2ax_{4},$\\
 &$
 [x_{5},x_{7}]=-x_{5}$ ($a\neq 0$) \\ \hline

$L_{4}^{ab}$ &
$[x_{1},x_{2}] = x_{3},\;
 [x_{1},x_{3}] = x_{4},\;
 [x_{1},x_{6}] = -x_{1},\;
 [x_{2},x_{6}] = -ax_{2},$\\
 &$
 [x_{2},x_{7}] = -bx_{2},\;
 [x_{3},x_{6}] = -(1+a)x_{3},\;
 [x_{3},x_{7}] = -bx_{3},$\\
 &$
 [x_{4},x_{6}] = -(2+a)x_{4},\;
 [x_{4},x_{7}] = -bx_{4},\;
 [x_{5},x_{7}] = -x_{5}$ ($b\neq 0$) \\ \hline

$L_{5}^{a}$ &
$[x_{1},x_{2}] = x_{3},\;
 [x_{1},x_{3}] = x_{4},\;
 [x_{1},x_{6}] = -x_{1},\;
 [x_{1},x_{7}] = -x_{2},$\\
 &$
 [x_{2},x_{6}] = -x_{2},\;
 [x_{3},x_{6}] = 2x_{3},\;
 [x_{4},x_{6}] = -3x_{4},\;
 [x_{5},x_{6}] = -a x_{5},$\\
 &$
 [x_{5},x_{7}] = -x_{5}$ \\ \hline

$L_{6}$ &
$[x_{1},x_{2}] = x_{3},\;
 [x_{1},x_{3}]=x_{4},\;
 [x_{1},x_{7}]=-x_{1}-x_{5},\;
 [x_{2},x_{6}]=-x_{2},$\\
 &$
 [x_{3},x_{6}]=-x_{3},\;
 [x_{3},x_{7}]=-x_{3},\;
 [x_{4},x_{6}]=-x_{4},\;
 [x_{4},x_{7}]=-2x_{4},$\\
 &$
 [x_{5},x_{7}]=-x_{5}$ \\ \hline

$L_{7}^{a}$ &
$[x_{1},x_{2}]=x_{3},\;
 [x_{1},x_{3}] = x_{4},\;
 [x_{1},x_{6}] = -x_{1},\;
 [x_{1},x_{7}] = -x_{5},$\\
 &$
 [x_{2},x_{6}] = -a x_{2},
 [x_{2},x_{7}] = - x_{2},
 [x_{6},x_{3}] = (1+a)x_{3},
 [x_{3},x_{7}] = -x_{3},$\\
 &$
 [x_{4},x_{6}] = -(2+a)x_{4},\;
 [x_{4},x_{7}] = -x_{4},\;
 [x_{5},x_{6}] = -x_{5}$ \\ \hline

$L_{8}^{a}$ &
$[x_{1},x_{2}] = x_{3},\;
 [x_{1},x_{3}]=x_{4},\;
 [x_{2},x_{6}] = -x_{2},\;
 [x_{2},x_{7}] = -x_{4},$\\
 &$
 [x_{3},x_{6}] = -x_{3},\;
 [x_{4},x_{6}] = -x_{4},\;
 [x_{5},x_{6}] = -a x_{5},\;
 [x_{5},x_{7}] = -x_{5}$ \\ \hline

$L_{9}$ &
$[x_{1},x_{2}] = x_{3},\;
 [x_{1},x_{3}] = x_{4},\;
 [x_{1},x_{7}] = -x_{1},\;
 [x_{2},x_{6}] = -x_{2},$\\
 &$
 [x_{2},x_{7}] = -x_{5},\;
 [x_{3},x_{6}] = -x_{3},\;
 [x_{3},x_{7}] = -x_{3},\;
 [x_{4},x_{6}] = -x_{4},$\\
 &$
 [x_{4},x_{7}] = -2x_{4},\;
 [x_{5},x_{6}] = -x_{5}$ \\ \hline

$L_{10}^{a}$ &
$[x_{1},x_{2}] = x_{3},\;
 [x_{1},x_{3}] = x_{4},\;
 [x_{1},x_{6}] = -x_{1},\;
 [x_{2},x_{6}] = -ax_{2},$\\
 &$
 [x_{2},x_{7}] = -x_{2}-x_{5},\;
 [x_{3},x_{6}]=-(1+a)x_{3},\;
 [x_{3},x_{7}] = -x_{3},$\\
 &$
 [x_{4},x_{6}] = -(2+a)x_{4},\;
 [x_{4},x_{7}] = -x_{4},\;
 [x_{6},x_{5}] = ax_{5},\;
 [x_{7},x_{5}] = x_{5}$ \\ \hline

$L_{11}$ &
$[x_{1},x_{2}] = x_{3},\;
 [x_{1},x_{3}] = x_{4},\;
 [x_{1},x_{7}] = -x_{1},\;
 [x_{2},x_{6}] = -x_{2},$\\
 &$
 [x_{3},x_{6}] = -x_{3},\;
 [x_{3},x_{7}] = -x_{3},\;
 [x_{4},x_{6}] = -x_{4},\;
 [x_{4},x_{7}] = -2x_{4},$\\
 &$
 [x_{5},x_{6}] = -x_{5},\;
 [x_{5},x_{7}] = -x_{4}-2x_{5}$ \\ \hline

$L_{12}^{a}$ &
$[x_{1},x_{2}] = x_{3},\;
 [x_{1},x_{3}] = x_{4},\;
 [x_{1},x_{6}] = -x_{1},\;
 [x_{2},x_{6}] = -a x_{2},$\\
 &$
 [x_{2},x_{7}] = -x_{2},\;
 [x_{3},x_{6}] = -(1+a)x_{3},\;
 [x_{3},x_{7}] = -x_{3},$\\
 &$
 [x_{4},x_{6}] = -(2+a)x_{4},\;
 [x_{4},x_{7}] = -x_{4},\;
 [x_{5},x_{6}] = -(2+a)x_{5},$\\
 &$
 [x_{5},x_{7}] = -x_{4}-x_{5}$ \\ \hline

\end{tabular}
\caption{Lie algebras listed in \cite[Table A.2]{Vu}.}
\end{table}

The types of these Lie algebras are presented in Table 2, following the application of the above algorithm.

%%%%%%%%
\begin{table}[h!]\label{T2}
\centering
\begin{tabular}{|p{8.9cm}|p{2.2cm}|}
\hline
\textbf{Lie algebras} & \textbf{Types} \\
\hline
$L_{1}$, $L_{2}$ & Mixed\\
\hline
$L_{3}^{a}$, $L_{4}^{ab}$, $L_{5}^{a}$, $L_{6}$, $L_{7}^{a}$, $L_{8}^{a}$, $L_{9}$, $L_{10}^{a}$, $L_{11}$, $L_{12}^{a}$ & Kronecker\\
\hline
\end{tabular}
\caption{Types of Lie algebras listed in \cite[Table A.2]{Vu}.}
\end{table}
%%%%%%%%%%%

%%%%%%%%%%%%%%%%%%%%%%%%%%%
\bibliographystyle{siam.bst}
\bibliography{references}

\end{document}